\documentclass{amsart}

\usepackage{amscd}
\usepackage{amsmath}
\usepackage{amssymb}
\usepackage{amsthm}
\usepackage{epsf}
\usepackage{float}
\usepackage{latexsym}
\usepackage{verbatim}
\usepackage[all, cmtip]{xy}
\usepackage{tikz}
\usetikzlibrary{positioning}
\usetikzlibrary{matrix}

\tikzstyle{bsq}=[rectangle, draw, thick, minimum width=1cm, minimum height=1cm]
\tikzstyle{bver}=[rectangle, draw, thick, minimum width=1cm, minimum height=2cm]
\tikzstyle{bhor}=[rectangle, draw, thick, minimum width=2cm, minimum height=1cm]

\newtheorem{theorem}{Theorem}[section]
\newtheorem{definition}[theorem]{Definition}
\newtheorem{lemma}[theorem]{Lemma}

\newtheorem{corollary}[theorem]{Corollary}
\newtheorem{proposition}[theorem]{Proposition}

\newtheorem{varexample}[theorem]{Example}
\newtheorem*{TropicalRR}{Tropical Riemann-Roch Theorem}
\theoremstyle{definition}
\newtheorem{remark}[theorem]{Remark}
\newtheorem*{BNthm}{Brill-Noether Theorem}
\newtheorem*{GPthm}{Gieseker-Petri Theorem}

\newcommand{\Spec}{\mathrm{Spec}\,}

\newcommand{\RR}{\mathbb{R}}
\newcommand{\ZZ}{\mathbb{Z}}

\newcommand{\cG}{\mathcal{G}}
\newcommand{\cL}{\mathcal{L}}
\newcommand{\cO}{\mathcal{O}}
\newcommand{\cW}{\mathcal{W}}
\newcommand{\cX}{\mathcal{X}}

\newcommand{\ocX}{\overline{\mathcal{X}}}

\newcommand{\ord}{\operatorname{ord}}
\newcommand{\Trop}{\operatorname{Trop}}
\newcommand{\trop}{\operatorname{trop}}
\newcommand{\ddiv}{\operatorname{div}}
\newcommand{\Div}{\operatorname{Div}}
\newcommand{\PL}{\operatorname{PL}}

\newcommand{\val}{\operatorname{val}}
\newcommand{\Pic}{\operatorname{Pic}}
\newcommand{\br}{\mathrm{br}}

\newcommand{\an}{\mathrm{an}}

\begin{document}
\title[Tropical independence I: Shapes of divisors and Gieseker-Petri]{Tropical independence I: Shapes of divisors and a proof of the {G}ieseker-{P}etri {T}heorem}
\author{David Jensen}
\author{Sam Payne}\thanks{Supported in part by NSF grants DMS--1068689 and CAREER DMS--1149054.}
\date{}
\bibliographystyle{alpha}

\maketitle

\begin{abstract}
We develop a framework to apply tropical and nonarchimedean analytic methods to multiplication maps for linear series on algebraic curves, studying degenerations of these multiplications maps when the special fiber is not of compact type.  As an application, we give a new proof of the Gieseker-Petri Theorem, including an explicit tropical criterion for a curve over a valued field to be Gieseker-Petri general.
\end{abstract}

\section{Introduction}

Classical Brill-Noether theory studies the schemes $\cG^r_d (X)$ parameterizing linear series of degree $d$ and rank $r$ on a smooth curve $X$ of genus $g$.  The Brill-Noether number $\rho(g,r,d) = g - (r+1)(g-d+r)$ is a naive dimension estimate for $\cG^r_d(X)$, and the following two fundamental results give the local structure of these schemes when the curve is general in its moduli space.

\begin{BNthm} \cite{GriffithsHarris80}
Let $X$ be a general curve of genus $g$.   Then $\cG^r_d (X)$ has pure dimension $\rho(g,r,d)$, if this is nonnegative, and is empty otherwise.
\end{BNthm}

\begin{GPthm} \cite{Gieseker82}
Let $X$ be a general curve of genus $g$.  Then $\cG^r_d (X)$ is smooth.
\end{GPthm}

\noindent The Zariski tangent space to $\cG^r_d (X)$ at a linear series $W \subset \cL(D_X)$ has dimension $\rho(g,r,d) + \dim \ker \mu_W$, where
\[
\mu_W : W \otimes \cL(K_X - D_X) \to \cL( K_X )
\]
is the adjoint multiplication map.  In particular, $\cG^r_d(X)$ is smooth of dimension $\rho(g,r,d)$ at a linear series $W$ if and only if the multiplication map $\mu_W$ is injective \cite[\S IV.4]{ACGH}.

Gieseker's original proof that $\mu_W$ is injective for all $W$ when $X$ is general involves a subtle degeneration argument.  Eisenbud and Harris developed a more systematic method for studying limits of linear series for one-parameter degenerations of curves in which the special fiber has compact type, and applied this theory to give a simpler proof of the Gieseker-Petri Theorem \cite{EisenbudHarris83c, EisenbudHarris86}.  Lazarsfeld gave another proof, without degenerations, using vector bundles on K3 surfaces \cite{Lazarsfeld86}.

Here, we give a new proof of the Gieseker-Petri Theorem, using a different class of degenerations, where the special fiber is not of compact type.  Our arguments are based in tropical geometry and Berkovich's theory of nonarchimedean analytic curves and their skeletons.

\bigskip

Let $\Gamma$ be a chain of $g$ loops connected by bridges, with generic edge lengths.

\begin{figure}[H] \label{Fig:ChainOfLoops}
\begin{tikzpicture}

\draw [ball color=black] (-1.7,-0.45) circle (0.55mm);
\draw (-1.95,-0.65) node {\footnotesize $v_1$};
\draw (-1.5,0) circle (0.5);
\draw (-1,0)--(0,0.5);
\draw [ball color=black] (-1,0) circle (0.55mm);
\draw (-0.85,0.3) node {\footnotesize $w_1$};
\draw (0.7,0.5) circle (0.7);
\draw (1.4,0.5)--(2,0.3);
\draw [ball color=black] (1.4,0.5) circle (0.55mm);
\draw [ball color=black] (0,0.5) circle (0.55mm);
\draw (-0.2,0.75) node {\footnotesize $v_2$};
\draw (2.6,0.3) circle (0.6);
\draw (3.2,0.3)--(3.87,0.6);
\draw [ball color=black] (2,0.3) circle (0.55mm);
\draw [ball color=black] (3.2,0.3) circle (0.55mm);
\draw [ball color=black] (3.87,0.6) circle (0.55mm);
\draw (4.5,0.3) circle (0.7);
\draw (5.16,0.5)--(5.9,0);
\draw (6.4,0) circle (0.5);
\draw [ball color=black] (5.16,0.5) circle (0.55mm);
\draw (5.48,0.74) node {\footnotesize $w_{g-1}$};
\draw [ball color=black] (5.9,0) circle (0.55mm);
\draw [ball color=black] (6.9,0) circle (0.55mm);
\draw (5.7,-.2) node {\footnotesize $v_g$};
\draw (7.3,-.2) node {\footnotesize $w_g$};

\draw [<->] (3.3,0.4) arc[radius = 0.715, start angle=10, end angle=170];
\draw [<->] (3.3,0.2) arc[radius = 0.715, start angle=-9, end angle=-173];

\draw (2.5,1.25) node {\footnotesize$\ell_i$};
\draw (2.75,-0.7) node {\footnotesize$m_i$};
\end{tikzpicture}
\caption{The graph $\Gamma$.}
\end{figure}
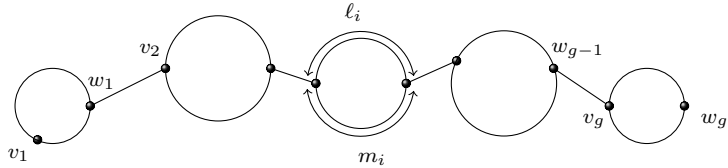

\noindent The genericity condition on edge lengths on the loops is the same as in \cite{tropicalBN}; we require that $\ell_i / m_i$ is not equal to the ratio of two positive integers whose sum is less than or equal to $2g-2$.

\begin{theorem}
\label{Thm:MainThm}
Let $X$ be a smooth projective curve of genus $g$ over a complete nonarchimedean field such that the minimal skeleton of the Berkovich analytic space $X^\an$ is isometric to $\Gamma$.  Then the multiplication map
$$\mu_W : W \otimes \cL(K_X - D_X) \to \cL(K_X ) $$
is injective for all linear series $W \subset \cL (D_X)$ on $X$.
\end{theorem}

\noindent There do exist such curves over valued fields of arbitrary pure or mixed characteristic.  This follows from the fact that the moduli space of tropical curves is the skeleton of the Deligne-Mumford compactification of the moduli space of curves \cite{acp}, and can also be proved by deformation theory, as in \cite[Appendix~B]{Baker08}.  The existence of Gieseker-Petri general curves over an arbitrary algebraically closed field then follows by standard arguments from scheme theory, using the fact that the coarse moduli space of curves is defined over $\Spec \ZZ$, as in \cite[Section~3]{tropicalBN}.  In particular, the Gieseker-Petri Theorem follows from Theorem~\ref{Thm:MainThm}, by standard arguments.
\bigskip

The proof of Theorem~\ref{Thm:MainThm} is essentially independent of the tropical proof of the Brill-Noether Theorem and does not involve the combinatorial classification of special divisors on a chain of loops from \cite{tropicalBN}.  (In Section~\ref{Sec:RhoZero}, we give a simplified proof in the special case where $\rho(g,r,d)$ is zero, which does use this classification; see Remark~\ref{Rmk:SpecialCase}.) Our approach involves not only the distribution of degrees over components of the special fiber, but also algebraic geometry over the residue field.  In particular, we use Thuillier's nonarchimedean analytic Poincar\'e-Lelong formula \cite{ThuillierThesis, BPR11}, which relates orders of vanishing at nodes in the special fiber of a semistable model to slopes of piecewise linear functions on the skeleton. The resulting interplay between tropical geometry and algebraic linear series is close in spirit to the important recent work of Amini and Baker on linear series on metrized complexes of curves \cite{AminiBaker12}, which was a source of inspiration.


\begin{remark}
The graph $\Gamma$ differs from the chain of loops studied in \cite{tropicalBN} only by the addition of bridges between the loops.  The tropical Jacobians of two graphs that differ by the addition or deletion of bridges are canonically isomorphic, and these isomorphisms respect the images of the Abel-Jacobi maps, so the Brill-Noether theory of $\Gamma$ is the same as that of the chain of loops.  See \cite{LPP12, Len14} for the basics of tropical Brill-Noether theory.

We do not need to introduce bridges for the case where $\rho(g,r,d)$ is zero; the arguments in Section~\ref{Sec:RhoZero} work equally well for a chain of loops without bridges.  However, when $\rho(g,r,d)$ is positive we need to relate the slopes of piecewise linear functions along the bridge edges to orders of vanishing at nodes in the special fiber, through the nonarchimedean Poincar\'e-Lelong formula, in order to produce bases for the algebraic linear series $\cL(D_X)$ with the required properties.  In particular, we do not know whether the conclusion of Theorem~\ref{Thm:MainThm} holds for chains of loops without bridges when $\rho(g,r,d)$ is positive.
\end{remark}

On the way to proving Theorem~\ref{Thm:MainThm}, we introduce some new techniques for working with tropical linear series and relating them to algebraic linear series.  In Section~\ref{Sec:Independence}, we present a notion of \emph{tropical independence}, which gives a sufficient condition for linear independence of rational functions on an algebraic curve $X$ in terms of the associated piecewise linear functions on the Berkovich skeleton of the analytic curve $X^\an$.   The key to applying such an independence condition is to produce well-understood piecewise linear functions on the skeleton that are not only in the tropical linear series, but are in fact tropicalizations of rational functions in a given algebraic linear series.  In the case where $\rho(g,r,d)$ is zero, the necessary piecewise linear functions come from tropicalizing a basis for the linear series and a basis for the adjoint linear series.  In this case, the piecewise linear functions are explicit and uniquely determined by the graph, and the proof that they all come from the algebraic linear series is essentially combinatorial.  (See Proposition~\ref{Prop:BasisDivisors}.)  When $\rho$ is positive, we have much less control over which tropical functions come from a given algebraic linear series.  In the general case, we work one loop at a time on the metric graph and use an existence argument from algebraic geometry, inspired by \cite[Lemma~1.2]{EisenbudHarris83c}.  (See Lemma~\ref{Lem:Basis}.)

One new insight on the tropical side is the importance of \emph{shapes} of effective divisors, expressed in terms of connected subsets that do or do not meet the divisor. When the metric graph is a chain of loops, a typical connected subset to consider would be a loop minus a single point.  See Sections~\ref{sec:Obstructions} and \ref{sec:CanonicalShapes}, along with the proofs of Theorems~\ref{Thm:RhoEqualsZero} and \ref{Thm:MainThm}, at the ends of Sections~\ref{Sec:RhoZero} and \ref{Section:MainResults}, respectively.

We also use a new \emph{patching} construction, gluing together tropicalizations of different rational functions in a fixed algebraic linear series on different parts of the graph, to arrive at a piecewise linear function in the corresponding tropical linear series that may or may not come from any linear combination of the original rational functions.  See the construction of $\theta$ at the beginning of the proof of Theorem~\ref{Thm:MainThm}.  The most delicate step in this construction is to ensure that no poles are introduced at the gluing points.

\bigskip

We now briefly sketch relations between the approach developed here, the classical theory of limit linear series, and the tropical theory of divisors on graphs.

\bigskip

Suppose $X$ is defined over a discretely valued field with valuation ring $R$, and let $L$ be a line bundle on $X$.  Consider a regular model $\cX$ over $\Spec R$ with general fiber $X$, in which the special fiber $\ocX$ is semistable with smooth components $\ocX_i$.  (By the semistable reduction theorem, such a model exists after a finite, totally ramified extension of the valued field.) The special fiber of this model has compact type, meaning that its Jacobian is compact, if and only if its dual graph is a tree.  In this case, for each component $\ocX_i$ there is a unique extension $\cL_i$ of the line bundle $L$ such that
\[
\deg\big(\cL_i|_{{\ocX}_j} \big)= \left\{ \begin{array}{ll} d & \mbox{ if } i = j, \\ 0 & \mbox{ otherwise.} \end{array} \right.
\]
Given a linear subspace $W \subset H^0(X,L)$ of degree $d$ and dimension $r+1$, the $R$-submodule $\cW_i \subset W$ consisting of sections that extend to $\cL_i$ is free of rank $r +1$, and restricts to a linear series of degree $d$ and dimension $r$ on $\ocX_i$.  The theory of limit linear series studies these distinguished linear series on the components of the special fiber, with special attention to their vanishing sequences at the nodes of $\ocX$.

In contrast, if $\ocX$ is not of compact type, then its dual graph is not a tree, and there is an obstruction to extending $L$ to a line bundle $\cL_i$ with degrees as above on the components of the special fiber.  This obstruction is given by an element in the component group of the N\'eron model of the Jacobian of $X$.

The theory of divisors on graphs follows a deep analogy between divisors on algebraic curves and the distributions of degrees of specializations of $L$ over the components of the special fiber. In this framework, one considers the dual graph whose vertices $v_i$ correspond to components $\ocX_i$ and whose edges correspond to nodes of $\ocX$.  Then an extension $\cL$ of $L$ to $\cX$ gives rise to a formal sum
\[
D_\cL = \sum_i \deg(\cL|_{\ocX_i}) v_i.
\]
which is considered as  a divisor on the graph.  Since the divisors arising from different specializations of  $L$ differ by a sequence of chip-firing moves, one studies the tropical Picard group parametrizing equivalence classes of divisors on the graph modulo the relation generated by chip-firing.  The tropical Jacobian, the degree zero part of this tropical Picard group, is canonically identified with the component group of the N\'{e}ron model of the Jacobian of $X$.

Baker's Specialization Lemma \cite{Baker08} says that a line bundle whose complete linear series has dimension $r$ can be specialized so that all degrees are nonnegative and the distribution of degrees dominates any given divisor of degree $r$ on the dual graph.  In other words, it has rank at least $r$ in the sense of \cite{BakerNorine07}.  Therefore, the specialization of any line bundle whose complete linear series has dimension at least $r$ lies in the tropical Brill-Noether locus parametrizing divisor classes of degree $d$ with rank at least $r$.  In \cite{tropicalBN}, a careful analysis of the Brill-Noether loci of the chain of loops shows that if a curve $X$ has a regular semistable model whose special fiber has this dual graph, then the curve must be Brill-Noether general, meaning that $\cG^r_d(X)$ has dimension $\rho(g,r,d)$ if this is non-negative, and is empty otherwise.   In particular, we get not only a new proof of the Brill-Noether Theorem, but an explicit and computationally verifiable sufficient condition for a curve to be Brill-Noether general, the existence of a regular semistable model whose special fiber has a particular dual graph.

\begin{remark}
This tropical proof of the Brill-Noether Theorem  can be reframed in the language of Berkovich's nonarchimedean analytic geometry to show that any curve of genus $g$ over a valued field whose skeleton is a chain of $g$ loops with generic edge lengths must be Brill-Noether general. Here, we follow this more general approach, with skeletons of analytifications in place of dual graphs of regular semistable models.  Similar arguments, combined with the basepoint-free pencil trick, lead to a proof of the Gieseker-Petri Theorem in the special case where $r = 1$ \cite{BJMNP}.
\end{remark}

\begin{remark}
In some ways, the tropical geometry of divisors on a chain of loops with generic edge lengths appears similar to the geometry of limit linear series on a chain of elliptic curves with generic attaching points.  As is well-known to experts in Brill-Noether theory, the theory of limit linear series on such curves gives a characteristic-free proof of the Brill-Noether and Gieseker-Petri theorems \cite{Osserman11, CLMTiB12}, and some steps in our approach, including Lemma~\ref{Lem:Basis} and Proposition~\ref{Prop:ChipsOnEachLoop}, can be viewed as tropical analogues of such arguments from classical algebraic geometry.

Other steps seem more difficult to translate. In the limit linear series proofs of Gieseker-Petri, both \cite{EisenbudHarris83c} and \cite{CLMTiB12} assume the multiplication map is not injective and use a degeneration argument to construct a divisor in $|K_X|$ of \emph{impossible degree}.  We assume the multiplication map is not injective and reach a contradiction by constructing an impossible divisor in $|K_\Gamma|$, but it is not the degree of this divisor that creates the contradiction.  Our argument relies on Proposition~\ref{Prop:Obstruction} and Lemma~\ref{Lem:CanonicalDivisorsSkipLoops} to show that the divisor has \emph{impossible shape}.

The relations to the geometry of the Deligne-Mumford compactification of $\mathcal{M}_g$ are also different.  Limit linear series arguments produce stable curves corresponding to points in the boundary of $\overline{\mathcal{M}}_g$ that are not in the closure of the Gieseker-Petri special locus, whereas the special fibers of our models are semistable, but necessarily unstable, and their stabilizations are always in the closure of the hyperelliptic locus.  (Limit linear series arguments may also involve semistable curves that are not stable, but the configurations of rational curves collapsed by stabilization tend to play an incidental role.  In sharp contrast, the precise combinatorial configurations of collapsed curves are essential to our arguments.)

It may still be tempting to try to interpret the tropical approach as a rephrasing or retranslation of classical degeneration arguments, at least in broad strokes, but there are fundamental obstacles to overcome.  As explained above, the data in our tropical arguments are in some sense strictly complementary to the data involved in limit linear series.  We work primarily in the component group of the N\'eron model of the Jacobian (or its analytic counterpart, the tropical Jacobian) whereas classical limit linear series are defined only in the case where this component group is trivial.  On the other hand, the limit linear series approach depends on computations in the compact part of the Jacobian of the special fiber, which is trivial in the cases we consider.

Finally, we note that even the Tropical Riemann-Roch Theorem has not been reinterpreted or reproved using classical algebraic geometry, despite multiple attempts.  Our proof of Gieseker-Petri uses this result in a crucial way, to control the shapes of effective canonical divisors (Lemma~\ref{Lem:CanonicalDivisorsSkipLoops}), so any satisfying interpretation of our argument in terms of classical degeneration methods should explain Tropical Riemann-Roch as well.
\end{remark}

\begin{remark}  \label{Rmk:SpecialCase}
In Section~\ref{Sec:RhoZero}, we give a simplified proof of Theorem~\ref{Thm:MainThm} in the special case where $\rho(g,r,d)$ is zero.  The simplified argument in this special case is essentially combinatorial, and relies on the classification of special divisors on a chain of loops in terms of rectangular tableaux \cite{tropicalBN} and the interpretation of adjunction in terms of transposition \cite{AMSW}.  It does not involve algebraic geometry over the residue field or the Poincar\'e-Lelong formula.

Although the guts of the argument are different, the overall structure of the proof by contradiction is the same as in the general case.  We assume that the multiplication map has nonzero kernel, deduce that certain carefully constructed collections of piecewise linear functions are tropically dependent, and use this dependence to produce a canonical divisor of impossible shape.  Although this section is not logically necessary, we believe that most readers will find it helpful to work through this special case first, as we did, before proceeding to the proof of Theorem~\ref{Thm:MainThm}.
\end{remark}

\subsection*{Acknowledgements}
We are grateful to Eric Katz and Joe Rabinoff for helpful conversations related to this work, to Dhruv Ranganathan for assistance with the illustrations, and to Matt Baker and the referee helpful comments on an earlier version of this draft that led to several improvements.  Important parts of this research were carried out during a week at Canada/USA Mathcamp in July 2013, supported by research in pairs grant NSF DMS-1135049.  We are grateful to the staff and students for their enthusiasm and warm hospitality.

\section{Background}
\label{Section:Background}

We briefly review the theory of divisors and divisor classes on metric graphs, along with relations to the classical theory of algebraic curves via Berkovich analytification and specialization to skeletons.  For further details and references, see \cite{BakerNorine07, Baker08, BPR11, AminiBaker12}.

\subsection{Divisors on graphs and Riemann-Roch}

Let $\Gamma$ be a metric graph.  A \emph{divisor} on $\Gamma$ is a finite formal sum
\[
D = a_1 v_1 + \cdots + a_s v_s,
\]
where the $v_i$ are points in $\Gamma$ and the coefficients $a_i$ are integers.  The \emph{degree} of a divisor is the sum of its coefficients
\[
\deg (D) = a_1 + \cdots + a_s,
\]
and a divisor is \emph{effective} if all of its coefficients are nonnegative.  We say that an effective divisor \emph{contains} a point $v_i$ if its coefficient $a_i$ is strictly positive. We will frequently consider questions about whether a given effective divisor $D$ contains at least one point in a connected subset $\Gamma' \subset \Gamma$.  See, for instance, Section~\ref{sec:Obstructions}.

Let $\PL(\Gamma)$ be the additive group of continuous piecewise linear functions $\psi$ with integer slopes on $\Gamma$. (Throughout, all of the piecewise linear functions that we consider have integer slopes.)  The \emph{order} of such a piecewise linear function $\psi$ at a point $v$ is the sum of its incoming slopes along edges containing $v$, and is denoted $\ord_v(\psi)$.  Note that $\ord_v(\psi)$ is zero for all but finitely many points $v$ in $\Gamma$, so
\[
\ddiv ( \psi ) = \sum_{v \in \Gamma} \ord_v ( \psi ) \, v,
\]
is a divisor.  A divisor is \emph{principal} if it is equal to $\ddiv(\psi)$ for some piecewise linear function $\psi$, and two divisors $D$ and $D'$ are \emph{equivalent} if $D - D'$ is principal.  Note that every principal divisor has degree zero, so the group $\Pic(\Gamma)$ of equivalence classes of divisors is graded by degree.

Let $D$ be a divisor on $\Gamma$.  The \emph{complete linear series} $|D|$ is the set of effective divisors on $\Gamma$ that are equivalent to $D$, and
\[
R(D) = \{ \psi \in \PL(\Gamma) \ | \ D + \ddiv(\psi) \mbox{ is effective}\}.
\]
These objects are closely analogous to the complete linear series of a divisor on an algebraic curve, and the vector space of rational functions with poles bounded by that divisor.
There is a natural surjective map from $R(D)$ to $| D |$ taking a piecewise linear function $\psi$ to $\ddiv ( \psi ) + D$, and two functions $\psi$ and $\psi'$ have the same image in $\vert D \vert$ if and only if $\psi - \psi'$ is constant.  The vector space structure on rational functions with bounded poles is analogous to the \emph{tropical module} structure on $R(D)$.  Addition in this tropical module is given by the pointwise minimum; if $\psi_0, \ldots, \psi_r$ are in $R(D)$ and $b_0, \ldots, b_r$ are real numbers, then the function $\theta$ given by
\[
\theta(v) = \min_j \big \{ \psi_j(v) + b_j \big \},
\]
is also in $R(D)$ \cite{HMY12}.

The \emph{rank} $r(D)$ is the largest integer $r$ such that $D-E$ is equivalent to an effective divisor for every effective divisor $E$ of degree $r$.  In other words, a divisor $D$ has rank at least $r$ if and only if its linear series contains divisors that dominate any effective divisor of degree $r$.  This invariant satisfies the following Riemann-Roch theorem with respect to the \emph{canonical divisor} $K_\Gamma = \sum_{v \in \Gamma} (\deg(v) - 2) \, v$.

\begin{TropicalRR}
\cite{BakerNorine07, GathmannKerber08, MikhalkinZharkov08}
Let $D$ be a divisor on a metric graph $\Gamma$ with first Betti number $g$.  Then
$$ r(D) - r(K_{\Gamma}-D) = \deg(D)-g+1 .$$
\end{TropicalRR}

\begin{remark}
Although it is closely analogous to the classical Riemann-Roch for curves, the Tropical Riemann-Roch Theorem has no known proof via algebraic geometry.  Indeed, neither of these results is known to imply the other.
\end{remark}

\subsection{Specialization of divisors from curves to graphs}

Throughout, we work over a fixed algebraically closed field $K$ that is complete with respect to a nontrivial valuation
\[
\val: K^* \rightarrow \RR.
\]
Let $R \subset K$ be the valuation ring, and let $\kappa$ be the residue field.

Let $X$ be an algebraic curve over $K$.  The underlying set of the Berkovich analytic space $X^\an$ consists of the closed points $X(K)$ together with the set of valuations on the function field $K(X)$ that extend the given valuation on $K$.  We write
\[
\val_y : K(X) \rightarrow \RR \cup \{ + \infty \}
\]
for the valuation corresponding to a point $y$ in $X^\an \smallsetminus X(K)$.

\begin{remark}
We treat the points in $X(K)$ differently, because they do not correspond to valuations on the function field $K(X)$.  Nevertheless, one can still study the closed points in terms of generalized valuations on rings, as follows.  If $U \subset X$ is any affine open neighborhood of a closed point $x \in X(K)$, then the map
\[
\val_x : \cO_X(U) \rightarrow \RR \cup \{ + \infty \}
\]
is a ring valuation.  Note that $\val_x$, unlike a valuation on a field, may take a nonzero element to $+ \infty$.
\end{remark}

 The topology on $X^\an$ is the weakest containing $U^\an$ for every Zariski open $U$ in $X$ and such that, for any $f \in \cO_X(U)$, the function taking $x \in U^\an$ to $\val_x(f)$ is continuous.

The points in $X(K)$ are called type-1 points of $X^\an$, and the remaining points in $X^\an \smallsetminus X(K)$ are classified into three more types according to the algebraic properties of the corresponding valuation on $K(X)$.  For our purposes, the most relevant points are type-2 points, the points $y$ such that the residue field of $K(X)$ with respect to $\val_y$ has transcendence degree 1 over $\kappa$.  We write $X_y$ for the smooth projective curve over the residue field of $K$ with this function field.

\begin{remark}
By passing to a spherically complete extension field whose valuation surjects onto $\RR$, one could assume that all points in $X^\an \smallsetminus X(K)$ are of type-2.
\end{remark}

Suppose $X$ is smooth and projective.  Then $X$ has a \emph{semistable vertex set}, a finite set of type-2 points whose complement is a disjoint union of a finite number of open annuli and an infinite number of open balls.  Each semistable vertex set $V \subset X^\an$ corresponds to a semistable model $\cX_V$ of $X$.  The normalized irreducible components of the special fiber $\ocX_V$ are naturally identified with the curves $X_y$, for $y \in V$, and the preimages of the nodes in $\ocX_V$ under specialization are the annuli in $X^\an \smallsetminus V$.  The annulus corresponding to a node where $X_y$ meets $X_{y'}$ contains a unique embedded open segement with endpoints $y$ and $y'$, whose length is the logarithmic modulus of the annulus. The union of these open segments together with $V$ is a closed connected metric graph embedded in $X^\an \smallsetminus X(K)$ with a natural metric.  We write $\Gamma_V$ for this metric graph, and call it the \emph{skeleton} of the semistable model $\cX_V$.  If $X$ has genus at least 2, which we may assume since the Gieseker-Petri Theorem is trivial for curves of genus 0 and 1, there is a unique minimal semistable vertex set in $X^\an$.  We write $\Gamma$ for the skeleton of this minimal semistable vertex set, and call it simply the \emph{skeleton} of $X^\an$.

Each connected component of $X^\an \smallsetminus \Gamma$ has a unique boundary point in $\Gamma$, and there is a canonical retraction to the skeleton
\[
X^\an \rightarrow \Gamma
\]
taking a connected component of $X^\an \smallsetminus \Gamma$ to its boundary point.  Restricting to $X(K)$ and extending linearly gives the tropicalization map on divisors
\[
\Trop: \Div(X) \rightarrow \Div(\Gamma).
\]
This map respects rational equivalence of divisors, as follows.

Let $f \in K(X)$ be a rational function. We write $\trop(f)$ for the real valued function on the skeleton $\Gamma$ given by $y \mapsto \val_y(f)$.  The function $\trop(f)$ is piecewise linear with integer slopes.  Furthermore, if $y$ is a type-2 point and $\trop(f)(y) = 0$, then the residue $\overline f_y$ is a nonzero rational function on $X_y$ whose slope along an edge incident to $y$ is the order of vanishing of $\overline f_y$ at the corresponding node.  This is the nonarchimedean Poincar\'{e}-Lelong Formula, due to Thuiller; see \cite{ThuillierThesis} and \cite[\S 5]{BPR11}.  One immediate consequence of this formula is that the tropical specialization map for rational functions
\[
\trop: K(X)^* \rightarrow \PL(\Gamma)
\]
is compatible with passing to principal divisors.  More precisely, for any nonzero rational function $f \in K(X)$, we have
\[
\Trop(\ddiv(f)) = \ddiv(\trop(f)).
\]
Therefore, the tropicalization map on divisors respects equivalences and descends to a natural map on Picard groups
\[
\Trop: \Pic(X) \rightarrow \Pic(\Gamma).
\]
Furthermore, since tropicalizations of effective divisors are effective, if $D_X$ is a divisor on $X$ and $f$ is a rational function in $\cL(D_X)$, then $\trop(f)$ is in $R(\Trop(D_X))$.  This leads to the following version of Baker's Specialization Lemma.

\begin{lemma}
\label{BakerSpecialization}
Let $D_X$ be a divisor on $X$.  Then $r( \Trop (D_X)) \geq r(D_X)$.
\end{lemma}

\noindent Here, the rank $r(D_X)$ is the dimension of the complete linear series of $D_X$ on $X$.

\begin{remark}
The Specialization Lemma and Riemann-Roch Theorem together imply that $\Trop(K_X) = K_\Gamma$, and hence tropicalization respects adjunction.  In other words, $\Trop(K_X - D_X) = K_\Gamma - \Trop(D_X)$.
\end{remark}

\begin{remark}
Note that $\trop (\cL(D_X))$ is often much smaller than $R(\Trop (D_X))$.  It is difficult in general to determine which piecewise linear functions in $R(\Trop(D_X))$ are tropicalizations of rational functions in $\cL(D_X)$.
\end{remark}

\section{Tropical Multiplication Maps}
\label{Section:Multiplication}

We now introduce a basic tropical lemma for studying linear dependence of rational functions and ranks of multiplication maps on linear series.

\subsection{Tropical independence}  \label{Sec:Independence}

Let $f_0, \ldots, f_r$ be rational functions on $X$.  Suppose $\{f_0, \ldots, f_r\}$ is linearly dependent, so there are constants $c_0, \ldots, c_r$ in $K$, not all zero, such that
\[
c_0 f_0 + \cdots + c_r f_r = 0.
\]
Then, for any point $v \in X^\an$, the minimum of the valuations
\[
\{ \val_v(c_0f_0), \ldots, \val_v(c_r f_r) \}
\]
must occur at least twice.  In particular, if $f_0, \ldots, f_r$ are linearly dependent in $K(X)$ then there are real numbers $b_0, \ldots, b_r$ such that the minimum of the piecewise linear functions $\{\trop(f_0) + b_0, \ldots, \trop(f_r) + b_r \}$ occurs at least twice at every point of the skeleton $\Gamma$.  Here, take $b_j = \val(c_j)$ if $c_j$ is nonzero, and otherwise make $b_j$ sufficiently large such that $\psi_j + b_j$ is never minimal.

\begin{definition}
A set of piecewise linear functions $\{ \psi_0, \ldots, \psi_r \}$ is \emph{tropically dependent} if there are real numbers $b_0, \ldots, b_r$ such that the minimum
\[
\min \{\psi_0(v) + b_0, \ldots, \psi_r(v) + b_r \}
\]
occurs at least twice at every point $v$ in $\Gamma$.
\end{definition}

\noindent If there are no such real numbers $b_0, \ldots, b_r$ then we say $\{ \psi_0, \ldots, \psi_r \}$ is \emph{tropically independent}.

\begin{lemma}
\label{Lem:MultiplicationMaps}
Let $D_X$ and $E_X$ be divisors on $X$, with $\{f_0, \ldots, f_r\}$ and $\{g_0, \ldots, g_s\}$ bases for $\cL(D_X)$ and $\cL(E_X)$, respectively.  If $\{\trop(f_i) + \trop(g_j)\}_{ij}$ is tropically independent then the multiplication map
\[
\mu: \cL(D_X) \otimes \cL(E_X) \rightarrow \cL(D_X+E_X)
\]
is injective.
\end{lemma}

\begin{proof}
The elementary tensors $f_i \otimes g_j$ form a  basis for $\cL(D_X) \otimes \cL(E_X)$.  The image of $f_i \otimes g_j$ under $\mu$ is the rational function $f_i g_j$, and these are linearly independent, since their tropicalizations are tropically independent.
\end{proof}

\begin{remark}
The main difficulty in applying this lemma is that one must prove the existence of rational functions in the algebraic linear series whose tropicalizations have the appropriate independence property.  Finding such piecewise linear functions in the tropical linear series is not enough.
\end{remark}

\subsection{Shapes of equivalent divisors} \label{sec:Obstructions}

Here we prove a technical proposition about how the tropical module structure on $R(D)$ is reflected in the shapes of divisors in $|D|$.  The proposition will be particularly useful when combined with our notion of tropical dependence of piecewise linear functions.

\begin{lemma}
\label{Lem:MinChips}
Let $D$ be a divisor on a metric graph $\Gamma$, with $\psi_0, \ldots, \psi_r$ piecewise linear functions in $R(D)$, and let
\[
\theta = \min \{ \psi_0, \ldots, \psi_r \}.
\]
Let $\Gamma_j \subset \Gamma$ be the closed set where $\theta = \psi_j$.  Then $\ddiv( \theta ) +D$ contains a point $v \in \Gamma_j$ if and only if $v$ is in either
\begin{enumerate}
\item  the divisor $\ddiv( \psi_j ) + D$, or
\item  the boundary of $\Gamma_j$.
\end{enumerate}
\end{lemma}

\begin{proof}
If $\psi_j$ agrees with $\theta$ on some open neighborhood of $v$, then $\ord_v ( \theta ) = \ord_v ( \psi_j )$, and hence $\ddiv ( \theta ) + D$ contains $v$ if and only if $\ddiv ( \psi_j ) + D$ does.  On the other hand, if $v$ is in the boundary of $\Gamma_j$ then there is an edge containing $v$ such that the incoming slope of $\theta$ along this edge is strictly greater than that of $\psi_j$, and the incoming slope of $\theta$ along any other edge containing $v$ must be at least as large as that of $\psi_j$.  By summing over all edges containing $v$ we find that $\ord_v ( \theta )$ is strictly greater than $\ord_v ( \psi_j )$.  Since $\ddiv \psi_j + D$ is effective, by hypothesis, it follows that the coefficient of $v$ in $\ddiv ( \theta ) + D$ is strictly positive, as required.
\end{proof}

\begin{proposition}
\label{Prop:Obstruction}
Let $D$ be a divisor on a metric graph $\Gamma$, with $\psi_0, \ldots, \psi_r$ in $R(D)$ and
\[
\theta = \min\{ \psi_0, \ldots, \psi_r \}.
\]
Let $\Gamma' \subset \Gamma$ be a connected subset, and suppose that $\ddiv( \psi_j ) + D$ contains a point in $\Gamma '$ for all $j$.  Then $\ddiv( \theta )+D$ also contains a point in $\Gamma'$.
\end{proposition}

\begin{proof}
Pick $j$ such that $\theta$ is equal to $\psi_j$ at some point in $\Gamma'$, and let
$$ \Gamma'_j = \{ v \in \Gamma'  \ | \  \theta (v) = \psi_j (v) \} .$$
If $\Gamma'_j$ is properly contained in $\Gamma'$, then its boundary is nonempty, since $\Gamma'$ is connected, and each of the boundary points is contained in $\ddiv (\theta) + D$, by Lemma~\ref{Lem:MinChips}.

Otherwise, if $\theta$ agrees with $\psi_j$ on all of $\Gamma'$, then $\ddiv(\theta) + D$ contains the points of $\ddiv(\psi_j) + D$ in $\Gamma'$, and the proposition follows.
\end{proof}

\section{The chain of loops with bridges}
\label{Section:TheGraph}

We now restrict attention to the specific graph $\Gamma$ shown in Figure~\ref{Fig:ChainOfLoops}, consisting of a chain of $g$ loops separated by bridges.  Throughout, we assume that the loops of $\Gamma$ have generic edge lengths in the same sense as in \cite{tropicalBN}, meaning that $\ell_i / m_i$ is never equal to the ratio of two positive integers whose sum is less than or equal to $2g-2$.

\subsection{Reduced divisors}

Fix a point $v \in \Gamma$.  Recall that an effective divisor $D$ is $v$-reduced if the multiset of distances from $v$ to points in $D$ is lexicographically minimal among all effective divisors equivalent to $D$.  Every effective divisor is equivalent to a unique $v$-reduced divisor, and the rank of a $v$-reduced divisor is bounded above by the coefficient of $v$.  In particular, if $D$ is a $v$-reduced divisor that does not contain $v$, then $r(D)$ is zero.  See \cite[Proposition 2.1]{Luo11}.

It is relatively straightforward to classify $v$-reduced divisors on $\Gamma$. We will only need the special case of $w_g$-reduced divisors.  For each $i$, let $\gamma_i$ be the $i$th loop minus $w_i$, the union of the two half-open edges $[v_i, w_i)$, and let $\br_i$ be the half-open bridge $[w_{i}, v_{i+1})$.  Note that $\Gamma$ decomposes as a disjoint union
\[
\Gamma = \gamma_1 \sqcup \br_1 \sqcup \cdots \sqcup \gamma_g \sqcup \{ w_g \},
\]
as shown.


\begin{figure}[H]
\begin{tikzpicture}
\matrix[column sep=0.5cm] {
\begin{scope}[baseline]
\draw [ball color=black] (-1.7,-0.45) circle (0.55mm);
\draw (-1.95,-0.65) node {\footnotesize $v_1$};
\draw (-1.5,0) circle (0.5);
\draw [ball color=white] (-1,0) circle (0.55mm);
\draw (-1.5,1.0) node {\footnotesize $\gamma_1$};
\end{scope}
&
\begin{scope}[grow=right,baseline]
\draw [ball color=black] (-1,0) circle (0.55mm);
\draw (-1,0)--(0,0.5);
\draw (-0.8,0.4) node {\footnotesize $w_1$};
\draw [ball color=white] (0,0.5) circle (0.55mm);
\draw (-0.5,1.2) node {\footnotesize $\br_1$};
\end{scope}
&
\begin{scope}[grow=right,baseline]
\draw node at (.5,0.48) {$\cdots$};
\end{scope}
&
\begin{scope}[grow=right,baseline]
\draw (0.7,0.5) circle (0.7);
\draw [ball color=white] (1.4,0.5) circle (0.55mm);
\draw [ball color=black] (0,0.5) circle (0.55mm);
\draw (0.7,1.5) node {\footnotesize $\gamma_i$};
\end{scope}
&
\begin{scope}[grow=right,baseline]
\draw (1.4,0.5)--(2,0.3);
\draw [ball color=white] (2,0.3) circle (0.55mm);
\draw [ball color=black] (1.4,0.5) circle (0.55mm);
\draw (1.7,1.1) node {\footnotesize $\br_i$};
\end{scope}
&
\begin{scope}[grow=right,baseline]
\draw node at (2,0.3) {$\cdots$};
\end{scope}
&
\begin{scope}[grow=right,baseline]
\draw (2.6,0.3) circle (0.6);
\draw [ball color=black] (2,0.3) circle (0.55mm);
\draw [ball color=white] (3.2,0.3) circle (0.55mm);
\draw (2.6,1.2) node {\footnotesize $\gamma_g$};
\end{scope}
&
\begin{scope}[grow=right, baseline]
\draw [ball color = black] (5.16,0.3) circle (0.55mm);
\draw node at (5.16,0.05) {\footnotesize $w_g$};
\end{scope}
\\};
\end{tikzpicture}
\caption{A decomposition of $\Gamma$.}
\end{figure}
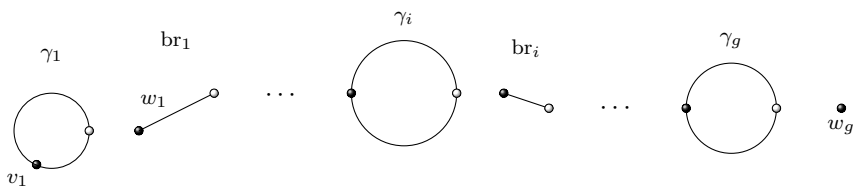

\begin{proposition}\label{Prop:ReducedDivisors}
An effective divisor $D$ is $w_g$-reduced if and only if it contains
\begin{enumerate}
\item no points in the bridges $\br_1, \ldots, \br_{g-1}$, and
\item at most one point in each cell $\gamma_1, \ldots, \gamma_g$.
\end{enumerate}
\end{proposition}

\begin{proof}
This is a straightforward application of Dhar's burning algorithm, as in \cite[Example~2.6]{tropicalBN}.
\end{proof}

\subsection{The shape of a canonical divisor} \label{sec:CanonicalShapes}

As mentioned in the introduction, our strategy is a proof by contradiction; we assume that a multiplication map has nonzero kernel and use Proposition~\ref{Prop:Obstruction} to construct a canonical divisor of \emph{impossible shape}.

The following basic lemma, which we state and prove but do not use, restricts the possibilities for the shape of a canonical divisor on an arbitrary graph.

\begin{lemma}
\label{Lem:CanonicalDivisorsOnAnyGraph}
Let $\Gamma'$ be a metric graph of genus $g$, let $e_1 , \ldots , e_g$ be disjoint open edges of $\Gamma'$ whose complement is a tree, and let $D$ be an effective divisor equivalent to $K_{\Gamma'}$.  Then at least one of the open edges $e_1, \ldots, e_g$ contains no point of $D$.
\end{lemma}

\begin{proof}
Suppose that each open edge $e_1, \ldots, e_g$ contains a point of $D$, let $p_i$ be a point in $e_i$, and let $D' = p_1 + \cdots + p_g$.  Since $K_{\Gamma'} - D'$ is effective, by construction, the Tropical Riemann-Roch Theorem says that $r(D')$ is at least 1.  However, Dhar's burning algorithm \cite{Dhar90} shows that $D'$ is $v$-reduced for any point $v$ in the complement of $e_1 \cup \cdots \cup e_g$.  Since $D'$ does not contain $v$, it follows that $r(D')$ is zero.
\end{proof}

\begin{remark}
Lemma~\ref{Lem:CanonicalDivisorsOnAnyGraph} also follows from the rigidity of effective representatives for classes in the relative interiors of top dimensional cells in the natural subdivision of $\Pic_g(\Gamma)$ into parallelotopes studied by An, Baker, Kuperberg, and Shokrieh \cite[Lemma~3.5]{ABKS}.
\end{remark}

On the chain of loops with bridges, we can use the classification of $w_g$-reduced divisors to refine the preceding lemma as follows.

\begin{lemma}
\label{Lem:CanonicalDivisorsSkipLoops}
Let $D$ be an effective divisor equivalent to $K_\Gamma$.  Then $D$ contains no point in at least one of the cells $\gamma_1, \ldots, \gamma_g$.
\end{lemma}

\begin{proof}
Suppose each cell $\gamma_1, \ldots, \gamma_g$ contains a point of $D$.  Let $p_i$ be a point of $D$ in $\gamma_i$, and let $D' = p_1 + \cdots + p_g$.  Then $K_\Gamma-D'$ is equivalent to an effective divisor, by construction, so the tropical Riemann-Roch Theorem says that $r(D')$ is at least 1.  However, $D'$ is $w_g$-reduced by Proposition~\ref{Prop:ReducedDivisors} and does not contain $w_g$, so $r(D')$ is zero.
\end{proof}

\begin{remark}
Note that the point $p_i$ in the proof of  Lemma~\ref{Lem:CanonicalDivisorsSkipLoops} may be equal to $v_i$ for some $2 \leq i \leq g$ in which case the complement of $\{p_1, \ldots, p_g \}$ is not a tree.  For this reason, the lemma does not follow from Lemma~\ref{Lem:CanonicalDivisorsOnAnyGraph}.  We use Lemma~\ref{Lem:CanonicalDivisorsSkipLoops} to obtain contradictions and prove our main results at the end of Sections~\ref{Sec:RhoZero} and \ref{Section:MainResults}.
\end{remark}

\section{Preliminaries for the proof of injectivity}

Let $X$ be a curve over $K$ with skeleton $\Gamma$, and let $D_X$ be a divisor of degree $d$ and rank $r$ on $X$.  To prove that $X$ is Gieseker-Petri general we must show that the multiplication map $\mu_W$ is injective for every linear subspace $W \subset \cL(D_X)$.  It clearly suffices to consider the case where $W = \cL(D_X)$.  In other words, we must show that
\[
\mu: \cL(D_X) \otimes \cL(K_X - D_X) \rightarrow \cL(K_X)
\]
is injective.

Given Lemma~\ref{Lem:MultiplicationMaps}, a natural strategy is to show that there are bases $\{f_i \}$ and $ \{ g_j \}$ for $\cL(D_X)$ and $\cL(K_X - D_X)$, respectively, such that the set of piecewise linear functions
\[
\{ \trop(f_i) + \trop(g_j) \}_{ij}
\]
is tropically independent.  We prove the existence of such a basis when the Brill-Noether number $\rho(g,r,d)$ is zero.  The following section, which treats this special case, is not logically necessary for the proof of Theorem~\ref{Thm:MainThm}.  However, the basic strategy that we use is the same as in the general case, only the details are simpler.

\begin{remark}
When $\rho(g,r,d)$ is positive, we do not know whether there are bases $\{f_i\}$ and $\{g_j\}$ for $\cL(D_X)$ and $\cL(K_X -D_X)$, respectively, such that $\{ \trop(f_i) + \trop(g_j) \}$ is tropically independent.
\end{remark}

\section{A special case: Brill-Noether number zero} \label{Sec:RhoZero}

The results of this sections are not used in the proof of Theorem~\ref{Thm:MainThm}, but working through this special case where $\rho(g,r,d)$ is zero before proceeding to the proof of the general case should be helpful for most readers.  An overview of the argument is as follows.

We start by assuming that the multiplication map has a kernel, and therefore the tropicalization of the image under $\mu$ of any basis for $\cL(D_X) \otimes \cL(K_X - D_X)$ is tropically dependent.  We use this tropical dependence together with Proposition~\ref{Prop:Obstruction} to construct a divisor in $|K_\Gamma|$ that violates Lemma~\ref{Lem:CanonicalDivisorsSkipLoops}, i.e., a canonical divisor of impossible shape.  When the Brill-Noether number is zero, the bases for $\cL(D_X)$ and $\cL(K_X - D_X)$ are explicit and canonically determined, and we only need to choose one basis for each.

Additional subtleties in the general case include the choice of $g$ different bases for $\cL(D_X)$ and $\cL(K_X - D_X)$, one for each loop in $\Gamma$, and the application of Poincar\'e-Lelong to control the slopes of tropicalizations along the bridges.  Furthermore, the bases are not explicit in the general case, but Lemma~\ref{Lem:Basis} gives the existence of bases with the required properties.

\begin{remark}
For a completely different tropical proof of the Gieseker-Petri Theorem in the case $\rho(g,r,d) = 0$, using lifting arguments instead of tropical independence, see \cite[Proposition~1.6]{LiftingDivisors}.
\end{remark}

Suppose $D_X$ is a divisor of degree $d$ and rank $r$ on $X$, with $\rho(g,r,d) = 0$, and let $D$ be the $v_1$-reduced divisor equivalent to $\Trop(D_X)$.  There are only finitely many $v_1$-reduced divisors of degree $d$ and rank $r$ on $\Gamma$, and they are explicitly classified in \cite{tropicalBN}.  These divisors correspond naturally and bijectively to the rectangular standard tableau with $(g-d+r)$ rows and $(r+1)$ columns.  Note that, since $\rho(g,r,d) = 0$, the genus $g$ factors as
\[
g = (r+1)(g-d+r).
\]
In particular, the entries in the tableau corresponding to $D$ are the integers $1, \ldots, g$.

Fix the tableau corresponding to $D$.  We label the columns from $0$ to $r$, and the rows from $0$ to $g-d + r -1$. The tableau determines a Dyck path, consisting of a series of points $p_0, \ldots, p_g$ in $\ZZ^r$, as follows.  We write $e_0, \ldots, e_{r-1}$ for the standard basis vectors on $\ZZ^r$.  The starting and ending point of the Dyck path is
\[
p_0 = p_g = (r, \ldots, 1),
\]
and the $i$th step $p_i - p_{i-1}$ is equal to
\begin{itemize}
\item the standard basis vector $e_j$ if $i$ appears in the $j$th column of the tableau, for $0 \leq j < r$, or
\item the vector $(-1, \ldots, -1)$ if $i$ appears in the last column.
\end{itemize}
The tableau properties exactly ensure that each $p_i$ lies in the open Weyl chamber $x_0 >  \cdots > x_{r-1} > 0$.  We write $p_i(j)$ for the $j$th coordinate of $p_i$.

The divisor $D$ can be recovered from the Dyck path as follows.  The coefficient of $v_1$ is $r$.  If $i$ appears in the $j$th column of the tableau, for $0 \leq j < r$, then $D$ contains the point on the $i$th loop at distance $p_{i-1}(j) m_i$ modulo $(\ell_i + m_i)$ counterclockwise from $w_i$ with coefficient $1$.  If $i$ appears in the last column of the tableau, then $D$ contains no point in the $i$th loop.

\begin{remark}
In this bijection, adjunction of divisors corresponds to transposition of tableaux \cite[Theorem~39]{AMSW}.  Therefore, the $v_1$-reduced divisor $E$ equivalent to $\Trop(K_X - D_X)$ is exactly the divisor corresponding to the transpose of the tableau for $D$.
\end{remark}


\begin{proposition}
\label{Prop:BasisDivisors}
For each integer $0 \leq j \leq r$, there is a unique divisor $D_j$ equivalent to $D$ such that $D_j - jv_1 - (r-j)w_g$ is effective.  Moreover, $\gamma_i$ contains no point of $D_j$ if and only if $i$ appears in the $j$th column of the tableau corresponding to $D$.
\end{proposition}

\begin{proof}
The divisor $D_r$ is exactly $D$.  The remaining divisors $D_j$, are constructed in the proof of \cite[Proposition~4.10]{tropicalBN}, by an explicit chip-firing procedure.  One takes a pile of $r-j$ chips from $v_1$ and moves it to the right.  The pile of chips changes size as it moves, and has $p_i(j)$ chips when it reaches $v_i$.  As the pile moves across the $i$th loop, there is a single chip left behind in the interior of one of the edges unless $i$ appears in the $j$th loop, in which case the $i$th loop is left empty.  When the pile reaches $w_g$, it has $p_g(j) = r-j$ chips.  Since $j$ chips were left at $v_1$ at the start of the procedure, $D_j - j v_1 - (r-j) w_g$ is effective.  To see that $D_j$ is the unique divisor equivalent to $D$ with this property, note that $D_j - j v_1 - (r-j) w_g$ does not move; it is effective and contains no points on the bridges or at the vertices, and hence is $v$-reduced for every $v$ in $\Gamma$.
\end{proof}

\noindent Similarly, for $0 \leq k \leq g-d+r -1$ there is a unique divisor $E_k$ equivalent to the $v_1$-reduced adjoint divisor $E$ such that $E_k - k v_1 - (g-d+r-1 -k) w_g$ is effective, and $\gamma_i$ contains no point of $E_k$ if and only if $i$ appears in the $k$th row of the tableau.

It follows that the $g$ divisors $D_j + E_k$ are distinct and correspond to the loops of $\Gamma$, as follows.

\begin{corollary}
\label{Cor:PartitionLoops}
The connected subset $\gamma_i \subset \Gamma$ contains no point of $D_j + E_k$ if and only if $i$ appears in the $j$th column and $k$th row of the tableau corresponding to $D$.
\end{corollary}

\begin{proposition}
\label{Prop:RhoZeroBasis}
There is a basis $f_0, \ldots, f_r$ for $\cL(D_X)$ such that
\[
\Trop(D_X + \ddiv(f_j)) = D_j.
\]
\end{proposition}

\begin{proof}
Let $x$ and $y$ be points in $X(K)$ specializing to $v_1$ and $w_g$, respectively.  Since $D_X$ has rank $r$, there is a rational function $f_j \in \cL(D_X)$ such that $D_X + \ddiv(f_j)$ contains $x$ with coefficient at least $j$ and $y$ with coefficient at least $r - j$.  Then $\Trop(D_X + \ddiv(f_j))$ is an effective divisor and contains $v_1$ and $w_g$ with coefficient at least $j$ and $r-j$, respectively.  By Proposition \ref{Prop:BasisDivisors}, $\Trop(D_X + \ddiv(f_j)$ must be equal to $D_j$.
\end{proof}

\noindent  Similarly, there is a basis $\{g_0, \ldots, g_{g-d+r-1} \}$ for $\cL(K_X - D_X)$ such that
\[
\Trop(K_X - D_X + \ddiv(g_k)) = E_k.
\]
We proceed to study the piecewise linear functions
\[
\phi_j = \trop(f_j) \mbox{ \ \ and \ \ } \psi_k = \trop(g_k).
\]
Note that $D + \ddiv(\phi_j) = D_j$ and $E + \ddiv(\psi_k) = E_k$, and this determines each $\phi_j$ and $\psi_k$ up to an additive constant, .

\begin{theorem}
\label{Thm:RhoEqualsZero}
The set of $g$ piecewise linear functions $\{\phi_j + \psi_k\}_{jk}$ is tropically independent.
\end{theorem}

\begin{proof}
Suppose that $\{\phi_j + \psi_k\}_{jk}$ is tropically dependent.  Then there exist real numbers $b_{jk}$ such that the minimum
\[
\theta = \min_{j,k} \{\phi_j + \psi_k + b_{jk} \}
\]
occurs at least twice at every point in $\Gamma$.  Note that $D + E + \ddiv(\theta)$ is an effective canonical divisor, since $R(D+E)$ is a tropical module and $D$ and $E$ are adjoint.

We claim that $D + E + \ddiv \theta$ contains a point in each $\gamma_i$.  Choose $j_0$ and $k_0$ such that $i$ appears in the $j_0$th column and $k_0$th row of the tableau corresponding to $D$.  Then Corollary~\ref{Cor:PartitionLoops} says that $D + E + \ddiv(\phi_j + \psi_k + b_{jk})$ contains a point in $\gamma_i$ for $(j,k) \neq (j_0, k_0)$. Also, since the minimum of $\{\phi_j +  \psi_k + +b_{jk} \}$ occurs at least twice at every point of $\Gamma$, we have
\[
\theta = \min_{(j,k) \neq (j_0,k_0)} \{ \phi_j + \psi_k + b_{jk} \}.
\]
Therefore, by Proposition~\ref{Prop:Obstruction}, the divisor $D + E + \ddiv(\theta)$ contains a point in $\gamma_i$, as claimed.

We have shown that $D + E + \ddiv(\theta)$ is an effective canonical divisor that contains a point in each of $\gamma_1, \ldots, \gamma_g$.  But this is impossible, by  Lemma~\ref{Lem:CanonicalDivisorsSkipLoops}.
\end{proof}

\section{Proof of Theorem~\ref{Thm:MainThm}}
\label{Section:MainResults}

As in the previous two sections, let $X$ be a smooth projective curve of genus $g$ over $K$ with skeleton $\Gamma$.  Since skeletons are invariant under base change with respect to extensions of algebraically closed valued fields, we can and do assume that $K$ is spherically complete.

\begin{remark}
Spherical completeness is equivalent to completeness for discretely valued fields, but stronger in general.  We use spherical completeness only in the proof of Lemma~\ref{Lem:Basis}, to ensure that normed $K$-vector spaces have orthogonal bases.
\end{remark}

Let $D_X$ be an effective divisor on $X$.  We must show that the multiplication map
\[
\mu: \cL(D_X) \otimes \cL(K_X - D_X) \rightarrow \cL(D_X)
\]
is injective.  This is trivial if $\cL(K_X - D_X)$ is zero, so we assume there is an effective divisor $E_X$ equivalent to $K_X - D_X$.  We may also assume $v_1$ and $w_g$ are type-2 points and choose type-2 points $w_0$ and $v_{g+1}$ in the connected components of $X^\an \smallsetminus \Gamma$ with boundary points $v_1$ and $w_g$, respectively.  Then
\[
V = \{v_1, \ldots, v_{g+1}, w_0, \ldots, w_g \}
\]
is a semistable vertex set, with skeleton $\Gamma_V \supset \Gamma$ as shown.

\begin{figure}[H]
\begin{tikzpicture}

\draw [ball color=black] (-3,0) circle (0.55mm);
\draw (-3.25,0.25) node {\footnotesize $w_0$};
\draw (-3,0)--(-2,0);

\draw [ball color=black] (-2,0) circle (0.55mm);
\draw (-2.2,-0.2) node {\footnotesize $v_1$};
\draw (-1.5,0) circle (0.5);
\draw (-1,0)--(0,0.5);
\draw [ball color=black] (-1,0) circle (0.55mm);
\draw (-0.85,0.3) node {\footnotesize $w_1$};
\draw (0.7,0.5) circle (0.7);
\draw (1.4,0.5)--(2,0.3);
\draw [ball color=black] (1.4,0.5) circle (0.55mm);
\draw [ball color=black] (0,0.5) circle (0.55mm);
\draw (-0.2,0.75) node {\footnotesize $v_2$};
\draw (2.6,0.3) circle (0.6);
\draw (3.2,0.3)--(3.87,0.6);
\draw [ball color=black] (2,0.3) circle (0.55mm);
\draw [ball color=black] (3.2,0.3) circle (0.55mm);
\draw [ball color=black] (3.87,0.6) circle (0.55mm);
\draw (4.5,0.3) circle (0.7);
\draw (5.16,0.5)--(5.9,0);
\draw (6.4,0) circle (0.5);
\draw [ball color=black] (5.16,0.5) circle (0.55mm);
\draw (5.48,0.74) node {\footnotesize $w_{g-1}$};
\draw [ball color=black] (5.9,0) circle (0.55mm);
\draw [ball color=black] (6.9,0) circle (0.55mm);
\draw (5.7,-.2) node {\footnotesize $v_g$};
\draw (7.3,-.2) node {\footnotesize $w_g$};

\draw [<->] (3.3,0.4) arc[radius = 0.715, start angle=10, end angle=170];
\draw [<->] (3.3,0.2) arc[radius = 0.715, start angle=-9, end angle=-173];

\draw (2.5,1.25) node {\footnotesize$\ell_i$};
\draw (2.75,-0.7) node {\footnotesize$m_i$};

\draw [ball color=black] (7.9,0) circle (0.55mm);
\draw (8.15,0.25) node {\footnotesize $v_{g+1}$};
\draw (6.9,0)--(7.9,0);

\end{tikzpicture}
\caption{The skeleton $\Gamma_V$.}
\end{figure}
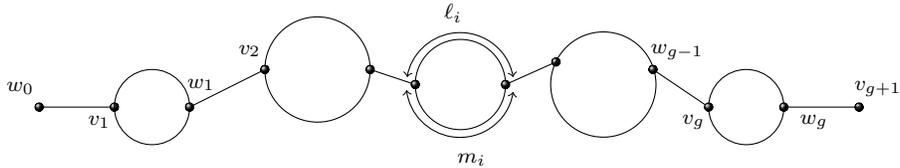

Let $\cX_V$ be the semistable model of $X$ associated to $V$, with $X_i$ the component of the special fiber $\ocX_V$ corresponding to $v_i$, and $x_i \in X_i$ the node corresponding to the edge $e_i = [w_{i-1}, v_i]$, for $1 \leq i \leq g + 1$.

Recall that the reduction of $f$ in $\kappa(X_i)^*$ is the residue of $a\hspace{-0.1em}f$ with respect to the valuation $\val_{v_i}$ on $K(X)$, where $a \in K^*$ is chosen such that $\val_{v_i}(a\hspace{-0.1em}f) = 0$ \cite{AminiBaker12}.  This reduction is defined only up to multiplication by elements of $\kappa^*$, but its order of vanishing at $x_i$ is independent of all choices.  Similarly, if $f_0, \ldots, f_r$ are rational functions in $K(X)^*$, then the $\kappa$-span of their reductions in $\kappa(X_i)$ is independent of all choices.  In particular, it makes sense to talk about whether these reductions are linearly independent.

\begin{lemma} \label{Lem:Basis}
Let $D_X$ be a divisor of rank $r$ on $X$. For each $1 \leq i \leq g$, there is a basis $f_0, \ldots, f_r$ for $\cL(D)$ such that
\begin{enumerate}
\item the reductions of $f_0, \ldots, f_r$ in $\kappa(X_i)$ have distinct orders of vanishing at $x_i$, and
\item the reductions of $f_0, \ldots, f_r$ in $\kappa(X_{i+1})$ are linearly independent.
\end{enumerate}
\end{lemma}

\begin{proof}
We consider $\cL(D_X)$ as a normed vector space over $K$, with respect to the norms $| \ |_i$ and $| \ |_{i+1}$ whose logarithms are $-\val(v_i)$ and $-\val(v_{i+1})$, respectively, and use the basic properties of nonarchimedean normed vector spaces developed in \cite[Chapter~2]{BGR84}.  Since $K$ is spherically complete, the vector space $\cL(D_X)$ is $K$-cartesian \cite[2.4.4.2]{BGR84}, and since $v_i$ and $v_{i+1}$ are type-2 points, the image of $\cL(D_X)$ under each of these norms is equal to the image of $K$ under its given norm.   Therefore, $\cL(D_X)$ is strictly $K$-cartesian \cite[2.5.1.2]{BGR84}, which means that all of its subspaces have orthonormal bases.  So, first choose an orthonormal basis for $\cL(D_X)$ with respect to $| \ |_i$.  The reductions of these basis elements are linearly independent \cite[2.5.1.3]{BGR84}, so we can take suitable combinations with coefficients in $R^*$ to ensure that they have distinct orders of vanishing at $x_i$.

Let $f_0, \ldots, f_r$ be a basis for $\cL(D_X)$ whose reductions in $\kappa(X_i)$ have strictly decreasing order of vanishing at $x_i$. Then, for each $j$, we can replace $f_j$ by a suitable linear combination of $f_0, \ldots, f_j$ that is orthogonal to the span of $f_0, \ldots, f_{j-1}$ with respect to $| \ |_{i+1}$.  This does not change the order of vanishing at $x_i$ of the reduction in $\kappa(X_i)$, but ensures that the reductions in $\kappa(X_{i+1})$ are linearly independent.
\end{proof}

\noindent This lemma, closely analogous to \cite[Lemma~1.2]{EisenbudHarris83c}, will be especially useful in combination with the following identity relating orders of vanishing of reductions of rational functions to the slopes of their tropicalizations.   For any piecewise linear function $\psi$ on $\Gamma_V$, we write $s_i(\psi)$ for the incoming slope of $\psi$ at $v_i$ along $e_i$.  Suppose $\psi = \trop(f)$ for some rational function $f$ in $K(X)^*$. Then Thuillier's nonarchimedean analytic Poincar\'e-Lelong formula \cite{ThuillierThesis, BPR11} says that $s_i(\trop(f))$ is the order of vanishing at $x_i$ of the reduction of $f$ in $\kappa(X_i)$.

Fix a basis $f_0, \ldots, f_r$ for $\cL(D_X)$ whose reductions in $\kappa(X_i)$ have distinct orders of vanishing at $x_i$, and whose reductions at $X_{i+1}$ are linearly independent.  Let $a_0, \ldots, a_r$ be constants in $K$.  Define
\[
\psi = \trop(a_0 f_0 + \cdots + a_r f_r).
\]
and
\[
\psi' = \min \{ \trop(f_0) + \val(a_0),  \ldots, \trop(f_r) + \val(a_r) \}.
\]
Note that
\[
\psi(v) \geq \psi'(v)
\]
for all $v$, with equality when $v$ is equal to $v_i$ or $v_{i+1}$.  This is because the reductions of the $a_j f_j$ in both $\kappa(X_i)$ and $\kappa(X_{i+1})$ are linearly independent.

\begin{proposition} \label{Prop:SumEqualsMin}
The piecewise linear functions $\psi$ and $\psi'$ are equal on some nonempty interval $(v, v_i) \subset e_i$.
\end{proposition}

\begin{proof}
The two functions $\psi$ and $\psi'$ agree at any point $v$ where the minimum of $\{ \trop(f_0)(v) + \val(a_0), \ldots, \trop(f_r)(v) + \val(a_r) \}$ occurs only once.  By construction, the reductions of $f_0, \ldots, f_r$ in $\kappa(X_i)$ have distinct orders of vanishing at $x_i$, so the Poincar\'e-Lelong formula says that $\trop(f_0), \ldots, \trop(f_r)$ have distinct incoming slopes at $v_i$ along $e_i$.  It follows that the minimum occurs only once on some open interval $(v,v_i)$, and $\psi$ and $\psi'$ agree on this interval.
\end{proof}

The final ingredient in our proof of Theorem~\ref{Thm:MainThm} is the following proposition relating slopes along bridges to shapes of divisors in a linear series on $\Gamma_V$.

\begin{proposition} \label{Prop:ChipsOnEachLoop}
Let $D$ be an effective divisor of degree at most $2g - 2$ on $\Gamma_V$, and let $\psi_0, \ldots, \psi_r \in R(D)$ be piecewise linear functions with distinct incoming slopes at $v_i$ along $e_i$, for some $1 \leq i \leq g$.  Then at most one of the divisors $D + \ddiv(\psi_0), \ldots, D + \ddiv(\psi_r)$ contains no point in $\gamma_i$.
\end{proposition}

\begin{proof}
Let $\Gamma'$ be the union of the $i$th loop together with a small closed subsegment of $[v, v_i] \subset [w_{i-1}, v_i]$ along which $\psi_0, \ldots, \psi_r$ all have constant slope.  We may choose $v$ sufficiently close to $v_i$ so that $D$ contains no points in $[v,v_i)$.  Let $D' = D|_{\Gamma'}$ and $\psi_j' = \psi_j|_{\Gamma'}$.  Note that the coefficient of $v$ in $\ddiv(\psi'_j)$ is $-s_i(\psi_j)$, and $D' + \ddiv(\psi'_j)$ agrees with $D + \ddiv(\psi_j)$ on $\gamma_i$.  We now show that at most one of the divisors $D' + \ddiv(\psi'_j)$ contains no point in $\gamma_i$.

Suppose $D' + \ddiv(\psi'_j)$ and $D' + \ddiv(\psi'_k)$ both contain no point in $\gamma_i$.  Then both of these divisors are supported at $v$ and $w_i$.  Subtracting one from the other, we find an equivalence of divisors
\[
(s_i(\psi_j) - s_i(\psi_k)) \, v \ \sim \ (s_i(\psi_j) - s_i(\psi_k)) \, w_i
\]
on $\Gamma'$.  Note that $s_i(\psi_j)$ is bounded above by the sum of the coefficients of $D$ at points to the left of $v_i$ and bounded below by minus the sum of its coefficients at $v_i$ and to the right.  Similarly, $-s_i(\psi_k)$ is bounded above by the sum of the coefficients of $D$ at $v_i$ and to the right, and bounded below by minus the sum of its coefficients at points to the left of $v_i$.  Therefore, $|s_i(\psi_j) - s_i(\psi_k)|$ is bounded by the degree of $D$. The equivalence above then implies that $\ell_i / m_i$ is a ratio of two positive integers whose sum is less than or equal to the degree of $D$, contradicting the genericity hypothesis on the edge lengths.
\end{proof}

\begin{proof}[Proof of Theorem~\ref{Thm:MainThm}]
Suppose the multiplication map
\[
\mu: \cL(D_X) \otimes \cL(E_X) \rightarrow \cL(K_X)
\]
has nonzero kernel.  For $1 \leq i \leq g$, let $\{f_0^i, \ldots, f_r^i\}$ be a basis for $\cL(D_X)$ consisting of rational functions whose reductions in $\kappa(X_i)$ have distinct orders of vanishing at $x_i$ and whose reductions in $\kappa(X_{i+1})$ are linearly independent.  Similarly, let $\{g_0^i, \ldots, g_{g-d+r-1}^i\}$ be a basis for $\cL(E_X)$ consisting of rational functions whose reductions in $\kappa(X_i)$ have distinct orders of vanishing at $x_i$ and whose reductions in $\kappa(X_{i+1})$ are linearly independent.

Fix an element in the kernel of $\mu$.  Then, for each $i$, we can express this element uniquely as a sum of elementary tensors
\[
\sum_{j,k} a^i_{j,k} f_j^i \otimes g_k^i.
\]
Define a piecewise linear function
\[
\theta_i =  \min_{j,k} \{ \trop (f_j^i) + \trop (g_k^i)  + \val (a^i_{j,k})\},
\]
and note that the minimum must occur at least twice at every point in $\Gamma_V$.

Replacing $\{f_0^i, \ldots, f_r^i \}$ by $\{ a\hspace{-0.1em}f_0^i, \ldots, a\hspace{-0.1em}f_r^i \}$ for some $a \in K^*$, we may assume that $\theta_i(v_{i+1}) = \theta_{i+1}(v_{i+1})$ for $1 \leq i < g$ and proceed by \emph{patching} these piecewise linear functions together.

Let $\theta$ be the unique continuous piecewise linear function on $\Gamma_V$ that agrees with $\theta_i$ between $v_{i}$ and $v_{i+1}$ for $1 \leq i \leq g$.  A priori, it is not clear whether $\theta$ is in the tropical linear series $R(D+E)$, where
\[
D = \Trop(D_X) \mbox{ \ \ and \ \ } E = \Trop(E_X).
\]
Nevertheless, we claim not only that $D + E + \ddiv(\theta)$ is effective but also that it contains a point in $\gamma_i$, for $1 \leq i \leq g$.  (Note that $\theta$ may or may not be the tropicalization of a rational function in $\cL(D_X + E_X)$.)

First we show that $D + E + \ddiv (\theta)$ is effective.  In the open subgraph between $v_i$ and $v_{i+1}$, the divisor $D + E + \ddiv(\theta)$ agrees with $D + E + \ddiv (\theta_i)$, which is effective because $R(D+E)$ is a tropical module that contains $\trop(f_j^i) + \trop(g_k^i)$ for all $j$ and $k$.  It remains to check that the coefficient of $v_i$ is nonnegative.  Since $D + E + \ddiv(\theta_i)$ is effective, it will suffice to show
\[
s_i(\theta_{i-1}) \geq s_i(\theta_i) .
\]
We prove this by changing coordinates in two steps, first replacing the basis $\{ f_j^i \}_j$ with $\{ f_j^{i-1} \}_j$ and then replacing the basis $\{g_k^i \}_k$ with $ \{g_k^{i-1} \}_k$.

Fix $k$, write
\[
\sum_j a^i_{j,k} f_j^i = \sum_j b_{j,k} f_j^{i-1},
\]
and define
\[
\theta' = \min_{j,k} \{ \trop(f_j^{i-1}) + \trop(g_j^i) + \val(b_{j,k}) \}.
\]
Note that
\[
\min_j \{ \trop(f_j^{i-1})(v_i) + \val(b_{j,k}) \} = \min_j \{ \trop(f_j^i)(v_i) + \val(a_{j,k}^i) \},
\]
since the reductions of both $\{f_j^{i} \}_j$ and $\{ f_j^{i-1} \}_j$ in $\kappa(X_i)$ are linearly independent.  By adding the constant $g_k^i(v_i)$ and taking the minimum over all $k$, we see that
\[
\theta'(v_i) = \theta(v_i).
\]

We now examine the slopes $s_i(\theta)$ and $s_i(\theta')$.  At any point $v$ on the edge between $w_{i-1}$ and $v_i$, we have
$$ \trop \big( \sum_j b_{j,k} f_j^{i-1} \big) (v) \geq \min_j  \big \{ \trop (b_{j,k} ) + \trop (f_j^{i-1} ) \big \} (v) .$$
Since this inequality holds with equality at $v_i$, it follows that
$$ s_i \big( \trop \big( \sum_j b_{j,k} f_j^{i-1} \big)\big) \leq s_i \big( \min_j \big \{ \trop (b_{j,k} ) + \trop (f_j^{i-1} ) \big \} \big ) . $$
Now Proposition~\ref{Prop:SumEqualsMin} tells us that, on some nonempty interval $(v, v_i) \subset e_i$,
$$ \trop \big( \sum_j b_{j,k} f_j^{i-1} \big) =  \min_j \big \{ \trop (a_{j,k}^i ) + \trop (f_j^i ) \big \} . $$
Taking the minimum over those values of $k$ for which $\min_j \{ \trop (a_{j,k}^i ) + \trop (f_j^i ) \} (v_i ) + \trop (g_k^i ) (v_i ) = \theta ( v_i )$, we see that
$$ s_i ( \theta_i ) \leq s_i ( \theta' ) . $$

A similar argument, fixing $j$ and replacing the basis $\{g^i_k \}$ with $\{g^{i-1}_k\}$ shows that $s_i(\theta') \leq s_{i}(\theta_{i-1})$, as required.  This proves that $D+ E + \ddiv(\theta)$ is effective.  It remains to show that $D + E + \ddiv(\theta)$ contains a point in each cell $\gamma_1, \ldots, \gamma_g$.

We now show that $D + E + \ddiv(\theta)$ contains a point in $\gamma_i$.  By Proposition~\ref{Prop:ChipsOnEachLoop}, there is at most one index $j$ such that $D + \ddiv(\trop(f_j^i))$ contains no point in $\gamma_i$.  Similarly, there is at most one index $k$ such that $E + \ddiv(\trop(g_k^i))$ contains no point in $\gamma_i$.  Call these indices $j_0$ and $k_0$, respectively, if they exist.  Note that, for $(j,k) \neq (j_0, k_0)$, the divisor $D + E + \ddiv(\trop(f_j^i)) + \ddiv(\trop(g_k^i))$ contains a point in $\gamma_i$.

By hypothesis, the minimum of the piecewise linear functions $\big \{ \trop(f_j^i)) + \ddiv(\trop(g_k^i)) + \val(a_{j,k}^i) \big \}$ occurs at least twice at every point, so
\[
\theta_i = \min_{(j,k) \neq (j_0, k_0)}  \big \{ \trop(f_j^i)) + \ddiv(\trop(g_k^i)) + \val(a_{j,k}^i) \big \}.
\]
Then Proposition~\ref{Prop:Obstruction} says that $D + E + \ddiv(\theta_i)$ contains a point in $\gamma_i$.  Now, $D + E + \ddiv(\theta)$ agrees with $D + E + \ddiv(\theta_i)$ on $\gamma_i \smallsetminus \{ v_i \}$.  Furthermore, since $s_i(\theta_i) \leq s_i(\theta_{i-1})$, the coefficient of $v_i$ in $D + E + \ddiv(\theta)$ is greater than or equal to the coefficient of $v_i$ in $D + E + \ddiv(\theta_i)$.  It follows that $D+ E + \ddiv(\theta)$ also contains a point in $\gamma_i$, as claimed.

Pushing forward the divisor $D+ E + \ddiv(\theta)$ under the natural contraction $\Gamma_V \rightarrow \Gamma$ gives an effective canonical divisor that contains a point in each cell $\gamma_1, \ldots, \gamma_g$.  But this is impossible, by Lemma~\ref{Lem:CanonicalDivisorsSkipLoops}.
\end{proof}

\bibliography{math}

\newcommand{\etalchar}[1]{$^{#1}$}
\begin{thebibliography}{CLMTiB12}

\bibitem[AB12]{AminiBaker12}
O.~Amini and M.~Baker.
\newblock Linear series on metrized complexes of algebraic curves.
\newblock To appear in Math. Ann. arXiv:1204.3508, 2012.

\bibitem[ABKS13]{ABKS}
Y.~An, M.~Baker, G.~Kuperberg, and F.~Shokrieh.
\newblock Canonical representatives for divisor classes on tropical curves and
  the {M}atrix-{T}ree {T}heorem.
\newblock arXiv:1304.4259v1, 2013.

\bibitem[ACGH85]{ACGH}
E.~Arbarello, M.~Cornalba, P.~A. Griffiths, and J.~Harris.
\newblock {\em Geometry of algebraic curves. {V}ol. {I}}, volume 267 of {\em
  Grundlehren der Mathematischen Wissenschaften}.
\newblock Springer-Verlag, New York, 1985.

\bibitem[ACP12]{acp}
D.~Abramovich, L.~Caporaso, and S.~Payne.
\newblock The tropicalization of the moduli space of curves.
\newblock To appear in Ann. Sci. \'{E}c. Norm. Sup. arXiv:1212.0373, 2012.

\bibitem[AMSW13]{AMSW}
R.~Agrawal, G.~Musiker, V.~Sotirov, and F.~Wei.
\newblock Involutions on standard {Y}oung tableaux and divisors on metric
  graphs.
\newblock {\em Electron. J. Combin.}, 20(3), 2013.

\bibitem[Bak08]{Baker08}
M.~Baker.
\newblock Specialization of linear systems from curves to graphs.
\newblock {\em Algebra Number Theory}, 2(6):613--653, 2008.

\bibitem[BGR84]{BGR84}
S.~Bosch, U.~G{\"u}ntzer, and R.~Remmert.
\newblock {\em Non-{A}rchimedean analysis}, volume 261 of {\em Grundlehren der
  Mathematischen Wissenschaften}.
\newblock Springer-Verlag, Berlin, 1984.

\bibitem[BJM{\etalchar{+}}12]{BJMNP}
V.~Baratham, D.~Jensen, C.~Mata, D.~Nguyen, and S.~Parekh.
\newblock Toward a tropical proof of the {G}ieseker-{P}etri {T}heorem.
\newblock To appear in Collectanea Math., 2012.

\bibitem[BN07]{BakerNorine07}
M.~Baker and S.~Norine.
\newblock {R}iemann-{R}och and {A}bel-{J}acobi theory on a finite graph.
\newblock {\em Adv. Math.}, 215(2):766--788, 2007.

\bibitem[BPR11]{BPR11}
M.~Baker, S.~Payne, and J.~Rabinoff.
\newblock Nonarchimedean geometry, tropicalization, and metrics on curves.
\newblock preprint, arXiv:1104.0320v1, 2011.

\bibitem[CDPR12]{tropicalBN}
F.~Cools, J.~Draisma, S.~Payne, and E.~Robeva.
\newblock A tropical proof of the {B}rill-{N}oether theorem.
\newblock {\em Adv. Math.}, 230(2):759--776, 2012.

\bibitem[CJP14]{LiftingDivisors}
D.~Cartwright, D.~Jensen, and S.~Payne.
\newblock Lifting divisors on a generic chain of loops.
\newblock To appear in Canad. Math. Bull. arXiv:1404.4001, 2014.

\bibitem[CLMTiB12]{CLMTiB12}
A.~Castorena, A.~L\'{o}pez~Mart\'{i}n, and M.~Teixidor~i Bigas.
\newblock Petri map for vector bundles near good bundles.
\newblock arXiv:1203.0983, 2012.

\bibitem[Dha90]{Dhar90}
D.~Dhar.
\newblock Self-organized critical state of sandpile automaton models.
\newblock {\em Phys. Rev. Lett.}, 64(14):1613--1616, 1990.

\bibitem[EH83]{EisenbudHarris83c}
D.~Eisenbud and J.~Harris.
\newblock A simpler proof of the {G}ieseker-{P}etri theorem on special
  divisors.
\newblock {\em Invent. Math.}, 74(2):269--280, 1983.

\bibitem[EH86]{EisenbudHarris86}
D.~Eisenbud and J.~Harris.
\newblock Limit linear series: basic theory.
\newblock {\em Invent. Math.}, 85(2):337--371, 1986.

\bibitem[GH80]{GriffithsHarris80}
P.~Griffiths and J.~Harris.
\newblock On the variety of special linear systems on a general algebraic
  curve.
\newblock {\em Duke Math. J.}, 47(1):233--272, 1980.

\bibitem[Gie82]{Gieseker82}
D.~Gieseker.
\newblock Stable curves and special divisors: {P}etri's conjecture.
\newblock {\em Invent. Math.}, 66(2):251--275, 1982.

\bibitem[GK08]{GathmannKerber08}
A.~Gathmann and M.~Kerber.
\newblock A {R}iemann-{R}och theorem in tropical geometry.
\newblock {\em Math. Z.}, 259(1):217--230, 2008.

\bibitem[HMY12]{HMY12}
C.~Haase, G.~Musiker, and J.~Yu.
\newblock Linear systems on tropical curves.
\newblock {\em Math. Z.}, 270(3-4):1111--1140, 2012.

\bibitem[Laz86]{Lazarsfeld86}
R.~Lazarsfeld.
\newblock Brill-{N}oether-{P}etri without degenerations.
\newblock {\em J. Differential Geom.}, 23(3):299--307, 1986.

\bibitem[Len14]{Len14}
Y.~Len.
\newblock The {B}rill--{N}oether rank of a tropical curve.
\newblock {\em J. Algebraic Combin.}, 40(3):841--860, 2014.

\bibitem[LPP12]{LPP12}
C.-M. Lim, S.~Payne, and N.~Potashnik.
\newblock A note on {B}rill-{N}oether theory and rank-determining sets for
  metric graphs.
\newblock {\em Int. Math. Res. Not. IMRN}, (23):5484--5504, 2012.

\bibitem[Luo11]{Luo11}
Y.~Luo.
\newblock Rank-determining sets of metric graphs.
\newblock {\em J. Combin. Theory Ser. A}, 118:1775--1793, 2011.

\bibitem[MZ08]{MikhalkinZharkov08}
G.~Mikhalkin and I.~Zharkov.
\newblock Tropical curves, their {J}acobians and theta functions.
\newblock In {\em Curves and abelian varieties}, volume 465 of {\em Contemp.
  Math.}, pages 203--230. Amer. Math. Soc., Providence, RI, 2008.

\bibitem[Oss11]{Osserman11}
B.~Osserman.
\newblock A simple characteristic-free proof of the {B}rill-{N}oether theorem.
\newblock {T}o appear in {B}ulletin of the {B}razilian {M}athematical {S}ociety
  (special issue in honor of {S}teven {K}leiman and {A}ron {S}imis).
  arXiv:1108.4967v1, 2011.

\bibitem[Thu05]{ThuillierThesis}
A.~Thuillier.
\newblock {\em Th{\'e}orie du potentiel sur ler courbes en g{\'e}om{\'e}trie
  analytique non archim{\'e}dienne. {A}pplications {\`a} la the{\'e}orie
  d'{A}rakelov}.
\newblock PhD thesis, University of Rennes, 2005.

\end{thebibliography}

\end{document}